\documentclass[a4paper,11pt]{amsproc}

\usepackage{xargs}
\usepackage{mathtools, todonotes}
\usepackage{amsmath,amssymb,amsfonts,amsthm}

\usepackage[hyphens]{url}
\usepackage[colorlinks, backref]{hyperref}
\usepackage{amsrefs}
\usepackage{enumerate, tikz-cd, nicefrac}

\usepackage{breqn}

\title[Equidistribution of word values]{On the asymptotic equidistribution of word values in symmetric groups}

\author{Vadim Alekseev}
\address{Vadim Alekseev, TU Dresden, 01062 Dresden, Germany}
\email{vadim.alekseev@tu-dresden.de}

\author{Jakob Schneider}
\address{Jakob Schneider, TU Dresden, 01062 Dresden, Germany}
\email{jakob.schneider@tu-dresden.de}

\author{Andreas Thom}
\address{Andreas Thom, TU Dresden, 01062 Dresden, Germany}
\email{andreas.thom@tu-dresden.de}

\theoremstyle{plain}
\newtheorem{theorem}{Theorem}[section]
\newtheorem{definition}[theorem]{Definition}
\newtheorem{proposition}[theorem]{Proposition}
\newtheorem{lemma}[theorem]{Lemma}
\newtheorem{corollary}[theorem]{Corollary}
\newtheorem{conjecture}{Conjecture}

\theoremstyle{definition}
\newtheorem{remark}[theorem]{Remark}
\newtheorem{example}[theorem]{Example}

\newcommand{\beq}{\begin{equation}}
\newcommand{\eeq}{\end{equation}}
\newcommand{\beqn}{\begin{equation*}}
\newcommand{\eeqn}{\end{equation*}}
\newcommand{\brq}{\begin{dmath}[compact]}
\newcommand{\erq}{\end{dmath}}
\newcommand{\brqn}{\begin{dmath*}[compact]}
\newcommand{\erqn}{\end{dmath*}}
\newcommand{\bag}{\begin{align}}
\newcommand{\eag}{\end{align}}
\newcommand{\bagn}{\begin{align*}}
\newcommand{\eagn}{\end{align*}}

\newcommand{\norm}[1]{\left\lvert#1\right\rvert}
\newcommand{\card}[1]{\left\lvert#1\right\rvert}
\newcommandx{\set}[2][2=\empty]{\{#1\ifx#2\empty\else\,|\,#2\fi\}}

\DeclareMathOperator{\fix}{fix}

\newcommand{\vertiii}[1]{{|\kern-0.2ex|\kern-0.2ex| #1 
    |\kern-0.2ex|\kern-0.2ex|}}

\begin{document}
\begin{abstract}
We prove an asymptotic equidistribution result for word values for words with constants in the symmetric group. We also speculate about simultaneous asymptotic equidistribution results for values of $d$-tuples of elements of $\mathbb F_d$.
\end{abstract}

\maketitle

\tableofcontents

\section{Introduction}

Denote the symmetric group on $n$ letters by $S_n$ and the free group on $r$ generators by $\mathbb F_r$. Every element $w \in \mathbb F_r$ naturally induces a word map $w \colon S_n^r \to S_n$ by evaluation. The study of the image of such word maps has many aspects and a long history, see \cite{MR2041227, MR3831276,MR3964506} and the references therein. A long-standing conjecture is that the image $w$ contains the alternating group, provided $w$ is not a power and $n$ is large enough. A much simpler metric version of this conjecture was shown to hold by the second and third author in \cite{MR4292938} for all non-trivial words. The natural metric in this context is the normalized Hamming metric
$$d(\sigma,\tau) = \frac1n |\{i \in \{1,\dots n\} \mid \sigma(i) \neq \tau(i)\}|, \quad \sigma,\tau \in S_n.$$ The main result of \cite{MR4292938} was the following theorem:
\begin{theorem}[\cite{MR4292938}] \label{thm:dense}
Let $w \in \mathbb F_r$ be non-trivial. There exist constants $c, \varepsilon>0$, such that for all $\sigma \in S_n$ we have
$d(\sigma,w(S_n^r))< cn^{-\varepsilon}.$
\end{theorem}

It is natural to study not just the image but the \emph{distribution of word values}, i.e., the push-forward of the normalized counting measure on $S_n^r$ to $S_n$. It is easy to see that for primitive $w$, the push-forward is the equidistribution on $S_n$ and the converse was proven in seminal work by Puder and Parzanchevski, \cite{MR3264763, MR4138706}. In general, it is conjectured that for particular words, the distribution approximates the equidistribution in the $L^1$-metric, see \cite{MR3997129, MR3034482, MR3449221}. As a consequence of our main result, we show that the distribution of word values for an arbitrary non-trivial word approximates the equidistribution, as $n$ increases, with respect to the associated Kantorovich-Rubinstein metric. The main difference between the Kantorovich-Rubinstein metric and the $L^1$-metric is that small changes in the Hamming metric are allowed to compare the measures, again allowing for more general positive results.

In fact, our results are more general and also apply to words with constants, provided that the so-called critical constants are fixed point free, see Theorem \ref{thm:main}. However, before stating our main result, we need some definitions: 

By a \emph{word with constants} in the group $G$ and in $r$ variables $x_1,\ldots,x_r$, we mean an element
$$
c_0x_{\iota(1)}^{\varepsilon(1)}c_1\cdots c_{l-1}x_{\iota(l)}^{\varepsilon(l)}c_l\in G\ast\mathbb{F}_r,
$$
of the free product, where $\mathbb{F}_r=\langle x_1,\ldots,x_r\rangle$ denotes the free group of rank $r$; $\varepsilon(j)\in\{\pm1\}$ ($j\in\{1,\ldots,l\}$), and $c_0,\ldots,c_l\in G$. We call $w$ \emph{reduced} if $\iota(j)=\iota(j+1)$ and $\varepsilon(j)=-\varepsilon(j+1)$ implies that $c_j \notin{Z}(G)$ for $j\in\{1,\ldots,l-1\}$. In this article, all words with constants are reduced. 
The \emph{length} of the (reduced) word $w$ is defined to be  $\lvert w\rvert\coloneqq l$. We call the constants $c_1,\ldots,c_{l-1}$ \emph{intermediate constants}. Moreover, we define the sets of indices
\begin{gather*}
J_0(w)\coloneqq\{j\in\{1,\ldots,l-1\}\mid\iota(j)\neq\iota(j+1)\},\\
J_+(w)\coloneqq\{j\in\{1,\ldots,l-1\}\mid\iota(j)=\iota(j+1)\text{ and }\varepsilon(j)=\varepsilon(j+1)\},\\
J_-(w)\coloneqq\{j\in\{1,\ldots,l-1\}\mid\iota(j)=\iota(j+1)\text{ and }\varepsilon(j)=-\varepsilon(j+1)\}.
\end{gather*}
The elements of $J_-(w)$ are called \emph{critical indices} and the elements $c_j$ with $j\in J_-(w)$ are called \emph{critical constants}. We call $w$ \emph{strong} if there are no critical indices. In other words, $w$ is strong if it does not contain a subword of the form $x^{-1} c x$ or of the form $x c x^{-1}$ for a variable $x$ and a constant $c$.
From $w$ we obtain the corresponding word map $w\colon G^r\to G$ by substitution, which is also denoted by $w$: Write $w(g_1,\ldots,g_r)\in G$ for the image of $w$ under the unique homomorphism $G\ast\mathbb{F}_r\to G$ which fixes $G$ elementwise and maps $x_i\mapsto g_i$ ($i\in\{1,\ldots,r\})$.
There is a natural homomorphism $\epsilon \colon G \ast \mathbb F_r \to \mathbb F_r$ that sends all elements of $G$ to the neutral element. We say that $w \in G \ast \mathbb F_r$ is \emph{non-singular} if $\epsilon(w) \neq 1.$

If $G$ is finite, then the normalized counting measure on $G^r$
can be pushed forward to $G$ via the map $w$, the resulting probability measure on $G$ is denoted by $\mu^{G}_{w} \in {\rm P}(G)$ and called the \emph{distribution of word values} of $w.$

Note that the Hamming metric is a bi-invariant metric and our normalization has the effect that it takes values in the interval $[0,1]$ for all $n.$
Using the normalized Hamming metric, we can also define a natural distance function on ${\rm P}(G)$, called the Kantorovich-Rubinstein metric:
$$d(\mu,\nu) = \min\left\{\int d(\sigma,\tau) \ d\gamma(\sigma,\tau) \mid \gamma \in {\rm P}(G \times G) \colon \pi_{1,*}(\gamma)=\mu, \pi_{2,*}(\gamma)=\nu  \right\},$$
for all $\mu,\nu \in {\rm P}(G).$ The  Kantorovich–Rubinstein metric measures the optimal way of transporting the mass of $\mu$ to the mass of $\nu$ with weights according to the normalized Hamming metric.

Given an arbitrary probability measure $\mu \in {\rm P}(S_n)$, we can make it conjugation invariant by defining $\Sigma(\mu)(A):= \frac{1}{n!} \sum_{\sigma \in S_n} \mu(A^{\sigma}).$ The equidistribution measure on $S_n$ is denoted by $\mu_{\rm unif}$.

We say that a word with constants $w \in S_n \ast \mathbb F_r$ is called \emph{regular} if $w$ is not conjugate to an element of $S_n$ and each critical constant of $w^2$ is fixed point free. If $w=c_0x_{\iota(1)}^{\varepsilon(1)}c_1\cdots c_{l-1}x_{\iota(l)}^{\varepsilon(l)}c_l$ with either $\iota(1) \neq \iota(l)$ or $\varepsilon(1) \neq - \varepsilon(l)$, this condition just says that critical constants of $w$ must be fixed point free.

\begin{theorem} \label{thm:main} Let $\varepsilon>0$. There exists a constant $c_\varepsilon>0$ such that the following holds: For any  $n \geq 1$ and $w \in S_n \ast \mathbb F_r$  a regular word with constants of length $l \geq 1$,  we have
$$d(\Sigma(\mu^{S_n}_{w}),\mu^{S_n}_{\rm unif}) \leq c_\varepsilon l^2 n^{-1/3 + \varepsilon}.$$
\end{theorem}

This result raises two natural questions that we cannot answer in full generality. First of all, is it necessary to symmetrize the push-forward or is it automatically asymptotically conjugation-invariant with respect to the Kantorovich-Rubinstein metric?  We conjecture that the answer is affirmative and discuss particular cases at the end of the next section.

\begin{conjecture} \label{conj:conj}
Let $w_n \in S_n \ast \mathbb F_r$ for $n \in \mathbb N$ be regular words with constants of bounded length. Then, we have $$\lim_{n \to \infty } d(\Sigma(\mu_{w_n}^{S_n}),\mu_{w_n}^{S_n}) = 0.$$
\end{conjecture}

Secondly, structurally similar results, see for example \cite{bodschneiderthom, MR3604379}, suggest that we do not need to assume anything on the critical constants in case the word is non-singular.

\begin{conjecture}
Let $w_n \in S_n \ast \mathbb F_r$ for $n \in \mathbb N$ be non-singular words with constants of bounded length. Then, we have
$$\lim_{n \to \infty} d(\Sigma(\mu^{S_n}_{w_n}),\mu^{S_n}_{\rm unif}) = 0.$$
\end{conjecture}

The proof of Theorem \ref{thm:main} is given in the next section while the proof of the crucial proposition is deferred to Section \ref{sec:proof}. The following corollary is already interesting:
\begin{corollary} \label{cor:noconstant}
Let $\varepsilon>0$. There exists a constant $c_\varepsilon>0$ such that the following holds:
If $w \in \mathbb F_r$ be a non-trivial word of length $l$, then we have
$$d(\mu^{S_n}_{w},\mu^{S_n}_{\rm unif}) \leq c_\varepsilon l^2 n^{-1/3 + \varepsilon}.$$
\end{corollary}

This result is a weak form of asymptotic equidistribution of word values, analogous to Theorem \ref{thm:dense}. Note however that neither implies the other.

\section{Invariant probability measures on $S_n$}
\label{sec:conj}

\begin{lemma} \label{lem:basic}
The equidistribution on a conjugacy class of a permutation with $d$ cycles to the equidistribution of long cycles has Kantorovich-Rubinstein distance equal to $d/n.$
\end{lemma}
\begin{proof}
For each permutation $\sigma \in S_n$ with $d$ cycles, write $\sigma$ in the cycle notation as
$\sigma = (s_1)(s_2)\cdots(s_d)$, where each $s_i$ is a list of natural numbers in $\{1,\dots,n\}$ such that each number appears in exactly one of the $s_i$'s. Observe that there are exactly $s:=d! \prod_i |s_i|$ ways of writing $\sigma$ in the cycle notation. Consider now the $n$-cycle $(s_1s_2\cdots s_d)$ and note that 
$$d(\sigma,(s_1s_2\cdots s_d))= d/n.$$  

We denote the conjugacy class of $\sigma$ by $[\sigma]$ and set $t := |[\sigma]|.$
Define a probability measure $\gamma$ on $S_n \times S_n$ giving weight $1/st$ to each pair $((t_1)\cdots(t_d),(t_1t_2\cdots t_d))$, where $(t_1)\cdots(t_d)$ is a way of writing $\tau \in [\sigma]$  in the cycle notation. It is obvious that $\pi_{1,*}(\gamma)$ is the equidistribution on $[\sigma]$ and that $\pi_{2,*}(\gamma)$ is the equidistribution on the set of $n$-cycles. Moreover,
$$\int_{S_n \times S_n} d(\sigma,\tau)\ d\gamma(\sigma,\tau) = d/n.$$ This gives an upper bound on the Kantorovich-Rubinstein distance. The lower bound is obvious since any permutation with $d$ cycles has normalized Hamming distance exactly $d/n$ to the set of $n$-cycles.
\end{proof}

\begin{lemma} \label{lem:ncycle}
The Kantorovich-Rubinstein distance of the equidistribution on $S_n$ to the equidistribution on the set of $n$-cycles in $S_n$ is at most $(\ln(n)+1)/n.$
\end{lemma}
\begin{proof}
This is a consequence of the previous lemma and the fact that the expected number of cycles of a random permutation in $S_n$ is $$1+\frac12 + \cdots \frac1n \leq \ln(n) +1.$$ This proves the claim.
\end{proof}

For $\sigma \in S_n$, we denote by $F_i(\sigma)$ the set of elements contained in an $i$-cycle and set $F_{\leq h}(\sigma) = \bigcup_{1 \leq i \leq h} F_i(\sigma)$. We also set $f_i(\sigma)= |F_i(\sigma)|$ and $f_{\leq h} = |F_{\leq h}(\sigma)|$.

Following the work of Eberhard-Jezernik \cite[Appendix A]{MR4359476} and extending techniques from \cite{MR4694588}, we prove the following proposition in Section \ref{sec:proof}. This result is the crucial computation in order to prove our main theorem.

\begin{proposition} \label{prop:ebjez}
There exists a constant $c>0$ such that the following holds for all $f,h,n \geq 0$. Let $w \in S_n \ast {\rm F}_r$ be regular of  length $l<c f^{1/2}/h$.  Then, we have
$$\mathbb{P}_{\sigma \in S^r_n}(f_{\leq h}(w(\sigma_1,\dots,\sigma_r)) \geq f)\leq \exp\left(-cf/(lh)^2\right).$$
\end{proposition}

Assuming this proposition, we are now ready to prove our main result.

\begin{proof}[Proof of Theorem \ref{thm:main}] Let $w \in S_n \ast \mathbb F_r$ of length $l$ be fixed. Let $\varepsilon>0$ be arbitrary. Consider Proposition \ref{prop:ebjez} with $h:=c n^{1/3 - \varepsilon}$ and $f:= l^2 n^{2/3}$. Note that $l < cf^{1/2}/h$  for large $n$ as required. Now, the conclusion is that apart from a set of $\mu_{w,S_n}$-measure bounded by $\exp( -n^{2\varepsilon}/c)$, the distribution of word values contains only permutations with the property that the number of points contained in cycles of length less or equal to $c n^{1/3 - \varepsilon}$ is bounded by $l^2 n^{2/3}$. Thus, for those permutations, the number of cycles is at most
$$l^2 n^{2/3} + \frac{n}{c n^{1/3 - \varepsilon}} \leq \left(l^2 + \frac1c\right) n^{2/3 + \varepsilon.}$$
We conclude from Lemma \ref{lem:basic} that the distance of $\Sigma(\mu_{w,S_n})$ to the equidistribution of $n$-cycles with respect to the Kantorovich-Rubinstein metric is at most 
$$\left(l^2 + \frac1c \right) n^{2/3 + \varepsilon}/n + \exp( -n^{2\varepsilon}/c)$$ and hence the distance to the equidistribution at most
$$\left(l^2 + \frac1c \right) n^{-1/3 + \varepsilon}+\exp( -n^{2\varepsilon}/c)+\frac{\ln(n)+1}{n}.$$  This proves the claim.
\end{proof}

Theorem \ref{thm:main} is most interesting in case $\mu_{w,S_n}$ is conjugation invariant, i.e., $\mu_{w,S_n}=\Sigma(\mu_{w,S_n})$, or in case there is at least some asymptotic form of invariance. Both happen naturally in interesting examples. First of all note that $\mu_{w,S_n}$ is invariant in case $w \in \mathbb F_r$. Secondly, for any $w \in S_n \ast \mathbb F_r$ it is automatically satisfied for $w'(x_1,\dots,x_{r+1})= x_{r+1} w(x_1,\dots,x_r) x_{r+1}^{-1} \in S_n \ast \mathbb F_{r+1}.$ 

Even more interesting, we may use Theorem \ref{thm:main} to show that distributions of word values are almost conjugation invariant in interesting non-trivial cases. For example, the distribution of word values of $w(x)=xgx$ is the same as the one for $w'(x)=x^2g$ (just changing the variable from $x$ to $g^{-1}x$). Since the distribution for $w''(x)=x^2$ is asymptotically equal to the equidistribution, the same holds for $w'$,  and hence for $w.$ We do not know how to deal with $w(x)=x^2gx$ or $w(x)=x^2gx^{-1}$ with a similar trick.

\section{Proof of the main estimate (Proposition \ref{prop:ebjez})}
\label{sec:proof}

\begin{lemma}\label{lem:eberhard}
Let $w\in S_n\ast\mathbb{F}_r$ of length $l$ as above and $h \geq 1$. Let $X\subseteq \{1,\dots,n\}$ be a $d$-element subset. Then
$$\mathbb{P}_{\sigma \in S^r_n}(X\subseteq F_{\leq h}(w(\sigma_1,\dots,\sigma_r))\leq\left(\frac{(n-\norm{w^h}_{\rm crit}+dlh)lh}{n-dlh} \right)^d.$$
\end{lemma}

\begin{proof} Let's first discuss the case $h=1$. We set $X= \{\alpha_1,\dots,\alpha_d\}.$ In order to study the event that $X \subseteq F_1(w(\sigma_1,\dots,\sigma_r))$, the following diagram is given to clarify various aspects of the proof:

\begin{gather*}
\alpha_1=\omega_{1,1}^{\varepsilon(1)}\stackrel{\sigma^{\varepsilon(1)}_{\iota(1)}}{\mapsto}
\omega_{1,1}^{-\varepsilon(1)}\stackrel{c_1}{\mapsto} \cdots
\stackrel{c_{l-1}}{\mapsto}
\omega_{1,l}^{\varepsilon(l)}\stackrel{\sigma^{\varepsilon(l)}_{\iota(l)}}{\mapsto}\omega_{1,l}^{-\varepsilon(l)}\stackrel{c_l}{\mapsto}\alpha_1\\
\vdots\\
\alpha_d=\omega_{d,1}^{\varepsilon(1)}\stackrel{\sigma^{\varepsilon(1)}_{\iota(1)}}{\mapsto}
\omega_{d,1}^{-\varepsilon(1)}\stackrel{c_1}{\mapsto} \cdots
\stackrel{c_{l-1}}{\mapsto}
\omega_{d,l}^{\varepsilon(l)}\stackrel{\sigma^{\varepsilon(l)}_{\iota(l)}}{\mapsto}\omega_{d,l}^{-\varepsilon(l)}\stackrel{c_l}{\mapsto}\alpha_d.
\end{gather*}

In an inductive procedure, we will define arrows that are connecting elements of $\{1,\dots,n\}$ and are colored by one of the colors $k\in\set{1,\ldots,r}$. The induction is carried out on the ordered set $\set{1,\ldots,d}\times\set{1,\ldots,l}$ in its lexicographical order, i.e.\ $(i',j')<(i,j)$ if and only if $j'<j$, or $j'=j$ and $i'<i$. 
In step $(i,j)$ we define the arrow 
$$
\alpha_{i,j}\colon\omega_{i,j}^{+1}\mapsto\omega_{i,j}^{-1}
$$ 
which is colored by the color $\iota(j)$. 

We collect the $k$-arrows defined up to step $(i,j)$ in the set $$\Lambda_{i,j,k}\coloneqq\set{\alpha_{i',j'}}[(i',j')\leq(i,j)\text{ and } \iota(j')=k].$$ In order that these $k$-arrows extend to permutations $\sigma_k$ for $k\in\set{1,\ldots,r}$, we must require that any two of these $k$-arrows $\alpha_{i',j'}$ and $\alpha_{i,j}$ for $(i',j')<(i,j)$ are either identical or  do not have the same source or target.
Also, as the $\sigma_k$ are supposed to be bijections, we must have that the $d$ elements $\omega_{1,j}^\varepsilon,\ldots,\omega_{d,j}^\varepsilon\in\Omega$ are pairwise distinct for $\varepsilon\in\set{\pm1}$ and $j\in\set{1,\ldots,l}$. 
We set $\Lambda_{i,j,k}^\varepsilon\coloneqq\set{\alpha^\varepsilon}[\alpha\in\Lambda_{i,j,k}]$, where $\alpha^\varepsilon$ is the $\varepsilon$-source of $\alpha$ for $\varepsilon\in\set{\pm1}$. 

In the inductive procedure, at step $(i,j)$, we only have to define the element $\omega_{i,j}^{-\varepsilon(j)}$. The rest is determined by the formula 
\begin{equation}
\label{formula1}\omega_{i,j+1}^{\varepsilon(j+1)}=\omega_{i,j}^{-\varepsilon(j)}.c_j,  \quad j\in\set{1,\ldots,l-1}, \end{equation} and $\omega_{i,1}^{\varepsilon(1)}\coloneqq\alpha_i$ for $i\in\set{1,\ldots,d}$. Note that in step $(i,j)$ of the induction, all $\omega_{i',j'}^{\pm1}$ for $(i',j')<(i,j)$ are already fixed, the $\omega_{i',j}^{\varepsilon(j)}$ for $i'\geq i$ are already defined, and the $\omega_{i',j+1}^{\varepsilon(j+1)}$ (for $i'<i$) are already chosen. We are to define $\omega_{i,j}^{-\varepsilon(j)}$ appropriately, which then will define $\omega_{i,j+1}^{\varepsilon(j+1)}$ by \eqref{formula1}. The tuple $(\omega_{i,j}^{\varepsilon(j)},\omega_{i,j}^{-\varepsilon(j)})$ will then determine the arrow $\alpha_{i,j}$.

There are a few cases to consider. First of all, if $\omega_{i,j}^{\varepsilon(j)}$ equals the $\varepsilon(j)$-source of an arrow $\alpha\in\Lambda_{i-1,j,\iota(j)}$, then we are forced to choose $\omega_{i,j}^{-\varepsilon(j)}$ as the $\varepsilon(j)$-target of $\alpha$, i.e.\ $\alpha=\alpha_{i,j}\colon\omega_{i,j}^{+1}\mapsto\omega_{i,j}^{-1}$ which then already lies in $\Lambda_{i-1,j,\iota(j)}$. Denote this by Case~A and call this phenomenon a \emph{forced choice}, in fact there is no real choice to be made.
The opposite case is denoted by Case~B. In this case $\omega_{i,j}^{-\varepsilon(j)}$ can be chosen almost arbitrarily and we say that there is a \emph{free choice} in step $(i,j)$. Without specifying this more precisely, this is a way of constructing $r$ random permutations. This specific way allows us to keep some control on the probability that the trajectory of some of the $d$ starting points $\alpha_1,\dots,\alpha_d$ will ever return to itself.

Throughout let $j\in\set{1,\ldots,l-1}$. We are interested in the probability that in step $(i,j+1)$ we have a forced choice while in step $(i,j)$ we still had a free choice.

We want to choose $\omega_{i,j}^{-\varepsilon(j)}$ and there is no arrow $\alpha\in\Lambda_{i-1,j,\iota(j)}$ with $\alpha^{\varepsilon(j)}=\omega_{i,j}^{\varepsilon(j)}$, otherwise we are in Case~A. Thus we must choose $\omega_{i,j}^{-\varepsilon(j)}$ from the set $\Omega\setminus\Lambda_{i-1,j,\iota(j)}^{-\varepsilon(j)}$ as all the arrows in $\Lambda_{i-1,j,\iota(j)}$ do not have the $\varepsilon(j)$-source $\omega_{i,j}^{\varepsilon(j)}$ and so cannot have the $\varepsilon(j)$-target $\omega_{i,j}^{-\varepsilon(j)}$. These are at least $n-dl$ many possible choices for $\omega_{i,j}^{-\varepsilon(j)}$. 

Now, we distinguish three subcases of Case~B according to the nature of the index $j$:

\emph{Case B.1: $j\in J_+(w)$; i.e.\ $\iota(j)=\iota(j+1)$ and $\varepsilon(j)=\varepsilon(j+1)$.} We look at the choice for $\omega_{i,j+1}^{\varepsilon(j+1)}$. Distinguish two subcases (a) and (b): At first we choose $\omega_{i,j}^{-\varepsilon(j)}.c_j=\omega_{i,j+1}^{\varepsilon(j+1)}$ such that (a) there is an arrow 
$$
\alpha\in\Lambda_{i-1,j,\iota(j+1)}=\Lambda_{i-1,j,\iota(j)}
$$ 
which has $\omega_{i,j+1}^{\varepsilon(j+1)}$ as its $\varepsilon(j+1)$-source or (b) there is an $\omega_{i',j}^{\varepsilon(j)}$ with $i'\geq i$ such that $\omega_{i,j+1}^{\varepsilon(j+1)}=\omega_{i,j}^{-\varepsilon(j)}.c_j=\omega_{i',j}^{\varepsilon(j)}$. Other cases cannot occur.  

In both cases (a) and (b) we run into a forced choice in step $(i,j+1)$. Indeed, in Case (a) we must choose $\omega_{i,j+1}^{-\varepsilon(j+1)}$ as the $\varepsilon(j+1)$-target of the arrow $\alpha=\alpha_{i,j+1}$ from above. In Case (b) until the step $(i,j+1)$, namely in step $(i',j)$, we will add another arrow $\alpha=\alpha_{i',j}\colon\omega_{i',j}^{+1}\mapsto\omega_{i',j}^{-1}$. Hence $\omega_{i,j+1}^{-\varepsilon(j+1)}=\omega_{i',j}^{-\varepsilon(j)}$. Note that there are at most 
$$
\card{\Lambda_{i-1,j,\iota(j+1)}^{\varepsilon(j)}\cup\set{\omega_{i,j}^{\varepsilon(j)},\ldots,\omega_{d,j}^{\varepsilon(j)}}}\leq dl
$$
many choices for $\alpha^{\varepsilon(j)}$ and points $\omega_{i',j}^{\varepsilon(j)}$ ($i'\in\set{i,\ldots,d}$). Hence the probability of running into a forced choice in step $(i,j+1)$ is bounded from above by $\frac{dl}{n-dl}$.

\emph{Case~B.2: $j\in J_0(w)$; i.e.\ $\iota(j)\neq\iota(j+1)$.} Here the proof is almost the same as the one of Case~B.1, we only do not need to take the $\omega_{i',j}^{\varepsilon(j)}$ into account ($i'\geq i$). We again get a factor $\frac{dl}{n-dl}$ for the probability that in step $(i,j+1)$ we have a fixed choice.

\emph{Case~B.3: $j\in J_-(w)$; i.e.\ $\iota(j)=\iota(j+1)$ and $\varepsilon(j)=-\varepsilon(j+1)$.} If we choose $\omega_{i,j}^{-\varepsilon(j)}\in\fix(c_j)$, then in step $(i,j+1)$ we will have the forced choice $\omega_{i,j+1}^{-\varepsilon(j+1)}=\omega_{i,j}^{\varepsilon(j)}$. These are $n-\norm{c_j}_{\rm H}\leq n-\norm{w}_{\rm crit}$ many possibilities.

In the opposite case, $\omega_{i,j+1}^{\varepsilon(j+1)}=\omega_{i,j+1}^{-\varepsilon(j)}=\omega_{i,j}^{-\varepsilon(j)}.c_j\neq\omega_{i,j}^{-\varepsilon(j)}$ can be the $\varepsilon(j)$-target (i.e.\ the $\varepsilon(j+1)$-source) of an arrow $\alpha\in\Lambda_{i-1,j,\iota(j+1)}$ different from $\alpha_{i,j}\colon\omega_{i,j}^{+1}\mapsto\omega_{i,j}^{-1}$. These are $d(j-1)+i-1$ many arrows. Assume there is no such $\alpha$, then there is a step $(i',j)$ with $i'>i$, such that $\omega_{i',j}^{\varepsilon(j)}$ is the $\varepsilon(j)$-source and $\omega_{i,j}^{-\varepsilon(j)}.c_j$ the $\varepsilon(j)$-target of the arrow $\alpha=\alpha_{i',j}$, which we define this way. These are $d-i$ many possibilities. Together with the above number of arrows, we get $d(j-1)+i-1+d-i=dj-1<dl$ many choices, each of which has probability $\frac{1}{n-dl}$.

Hence there is a probability of $\frac{\card{\fix(c_j)}+dl}{n-dl}\leq\frac{n-\norm{w}_{\rm crit}+dl}{n-dl}$ to arrive at a forced choice in this subcase.

\emph{Case~C: $j=l$.} Then we have the obstruction that $\omega_{i,l}^{-\varepsilon(l)}.c_l=\alpha_i$. The possibilities to choose $\omega_{i,l}^{-\varepsilon(l)}$ are bounded from below by $n-dl$ unless $(i,l)$ is a forced choice. Thus we get a factor of $\frac{1}{n-dl}$

We proceed with the final counting argument. We walk through all steps $(i,j)$ in lexicographical order to bound the probability that all $\alpha_i$ for $i\in\set{1,\ldots,d}$ are fixed by $w$. In this case, each line starts with a free choice and either (i) runs into a forced choice after finitely many free choices or (ii) continues with free choices until the last choice is according to Case C. In total there are at most $l^d$ cases according to the position of the first forced choice. In Case~(i), the occurrence of the first forced choice gives a factor bounded by $\frac{n-\norm{w}_{\rm crit}+dl}{n-dl}$. In Case~(ii), we get a factor $\frac{1}{n-dl}\leq\frac{n-\norm{w}_{\rm crit}+dl}{n-dl}$. In total, we get a bound on the probability of the form $$\left(\frac{n-\norm{w}_{\rm crit}+dl}{n-dl}\right)^d\cdot l^d$$ by a simple union bound. This completes the proof in the case $h=1$.

In the general case $h \geq 2$, we apply the above reasoning to $w^h$ and we need to study the probability that along the trajectory of length $lh$ we return to the starting point after $le$ steps for some $1 \leq e \leq h.$ Now, if we return on some trajectory after $le$ steps for $1 \leq e <h$, the next choice is definitely forced. Hence, we are back to the argument above.
\end{proof}

\begin{corollary}
Let $w\in S_n\ast\mathbb{F}_r$ be regular of length $l$. Let $h \geq 1$ and let $X\subseteq \{1,\dots,n\}$ be a $d$-element subset. Then
$$\mathbb{P}_{\sigma \in S^r_n}(X\subseteq F_{\leq h}(w(\sigma_1,\dots,\sigma_r)))\leq\left(\frac{d(lh)^2}{n-dlh} \right)^d.$$
\end{corollary}

The proof of Proposition \ref{prop:ebjez} follows now from the previous corollary exactly as in the proof of \cite{MR4359476}*{Theorem A.4}. 

\section{Invariant probability measures on $S_n^d$}

In this section we want to discuss the natural generalization of the previous results to a study of the simultaneous behavior of multiple words with constants. In order to explain the basic questions, let's consider the assignment $(x,y) \mapsto (x,y^2xy^{-1})$. We can think of this assignment as a homomorphism $\mathbb F_2 \to \mathbb F_2$ as well as inducing a natural map on $S_n \times S_n \to S_n \times S_n$ by evaluation for each $n \in \mathbb N$. 
We intend to study two basic notions of equidistribution of the associated word values.

\subsection{Sofic approximations of the free group}

Consider $w_1,\dots,w_d \in \mathbb F_r$. Note that the measures $(w_1,\dots,w_d)_*((\mu^{S_n}_{\rm unif})^{\otimes r})$ are invariant under the diagonal $S_n$-action by conjugation. We denote the set of measures $\mu \in {\rm P}(S_n^d)$ that are invariant under the diagonal conjugation action by ${\rm P}(S_n^d)^{S_n}$.

A first natural condition to study for a sequence $(\mu_n)_n$ of probability measures on $S_n^d$ is as follows:

\begin{definition} \label{def:sofic}
We say that a sequence of probability measures $(\mu_n)_n$ with $\mu_n \in {\rm P}(S_n^d)^{S_n}$ is a \emph{sofic approximation of the free group} if for all non-trivial $w \in \mathbb F_d$ the probability of the tail event 
$$\lim_{n \to \infty} \frac{1}{n}f_1(w(\sigma_{1,n},\dots,\sigma_{d,n})) = 0$$
is equal to $1$.
\end{definition}

Since there are only countably many words, the order of quantifiers in the preceding definition is irrelevant.

\begin{proposition}
Suppose $w_1,\dots,w_d \in \mathbb F_r$ and consider
$$\mu_n:=(w_1,\dots,w_d)_*((\mu^{S_n}_{\rm unif})^{\otimes r}).$$
The sequence $(\mu_n)_n$ is a sofic approximation of the free group if and only if $w_1,\dots,w_d$ generate a free subgroup of rank $d.$
\end{proposition}
\begin{proof}
Note that $w_*(\mu_n) = v_*((\mu^{S_n}_{\rm unif})^{\otimes r})$ for $v=w(w_1,\dots,w_d).$ Hence, the result is again a consequence of Proposition \ref{prop:ebjez} for $h=1$, a straightforward Borel-Cantelli argument and the observation that $v$ is non-trivial for all non-trivial $w$ if and only if the $w_1,\dots,w_d$ generate a free subgroup of rank $d$.
\end{proof}

\subsection{Simultaneous equidistribution}

It is natural to wonder if the distribution of word values in $S_n \times S_n$ approximates the uniform measure in the Kantorovich-Rubinstein metric on ${\rm P}(S_n \times S_n)$ with respect to the product metric of the Hamming metrics.
In order to make this precise, we fix the $\ell^1$-product-metric on $S_n^d$:
$$d((\sigma_1,\dots,\sigma_d),(\tau_1,\dots,\tau_d))= \sum_{i=1}^d d(\sigma_i,\tau_i)$$
and denote the associated Kantorovich-Rubinstein metric on the space of probability measures ${\rm P}(S_n^d)$ also by $d.$

\begin{definition}
We say that a $d$-tuple of words $(w_1,\dots w_d) \in \mathbb F_r^d$ is \emph{simultaneously asymptotically equidistributed}, or \emph{s.a.e.}\ for brevity, if 
$$\lim_{n \to \infty} d\left((w_1,\dots,w_d)_*((\mu^{S_n}_{\rm unif})^{\otimes r}), (\mu^{S_n}_{\rm unif})^{\otimes d} \right) =0,$$
where $d$ denotes the Kantorovich-Rubinstein-metric on the set of probability measures on $S_n^d.$
\end{definition}

We will need the following basic lemma:
\begin{lemma} \label{lem:counting} Let $d \in \mathbb N$. There exists a constant $c>0$, such that for all $n \in \mathbb N$, the following holds:
If $\mu \in {\rm P}(S_n^d)$ satisfies $d(\mu,(\mu^{S_n}_{\rm unif})^{\otimes d})\leq \varepsilon$,
then $${\rm supp}(\mu) \geq c (n!)^{(1- \varepsilon^{1/2})d}.$$
\end{lemma}
\begin{proof}
Another natural metric on ${\rm P}(S_n^d)$ is the Lévy–Prokhorov metric
$$\pi(\mu,\nu):= \inf\{\varepsilon>0 \mid \mu(A) \leq \nu(A^{\varepsilon}) + \varepsilon, \nu(A) \leq \mu(A^{\varepsilon}) + \varepsilon\}.$$
It is a basic fact that $\pi(\mu,\nu) \leq d(\mu,\nu)^{1/2}.$ Now, $d(\mu,(\mu^{S_n}_{\rm unif})^{\otimes d})) \leq \varepsilon$ implies that
$\pi(\mu,(\mu^{S_n}_{\rm unif})^{\otimes d})) \leq \varepsilon^{1/2}$ and hence
$$1 =\mu({\rm supp}(\mu)) \leq \frac{|{\rm supp}(\mu)^{\varepsilon^{1/2}}|}{(n!)^d} + \varepsilon^{1/2}.$$

It remains to bound the size of Hamming balls in $S_n$.
Let $B_\delta(\mathrm{id}) := \{\sigma \in S_n : d(\sigma,\mathrm{id}) \le \delta\}$.
Then every $\sigma \in B_\delta(\mathrm{id})$ moves at most $\delta n$ points, so
\[
|B_\delta(\mathrm{id})|
\le \sum_{k \le \delta n} \binom{n}{k} k!
\le \sum_{k \le \delta n} n^k
\le (\delta n+1) n^{\delta n}.
\]
Using Stirling's formula, this implies
\[
|B_\delta(\mathrm{id})| \le c' (n!)^{\delta}
\]
for some universal constant $c'>0$ and all sufficiently small $\delta$.
Since all Hamming balls have the same size, the claim follows.
Thus, we obtain $|{\rm supp}(\mu)^{\varepsilon^{1/2}}| \leq c'|{\rm supp}(\mu)| (n!)^{d\varepsilon^{1/2}}$ and this implies the claim.
\end{proof}

\begin{example}
If a sequence $(\sigma_n,\tau_n)_n$ with $ (\sigma_n,\tau_n)\in S_n^2$ is a (deterministic) sofic approximation of the free group, then the sequence of measures $(\mu_n)_n$ with $\mu_n \in {\rm P}(S_n^2)$, which is giving a random conjugate of this sequence, is a sofic approximation of the free group in the sense of Definition \ref{def:sofic}. However, the measure $\mu_n$ is supported on at most $n!$ points and thus cannot be close to the uniform distribution on $S_n^2$ by the previous lemma.
\end{example}

One may wonder about sufficient conditions on a $d$-tuple $(w_1,\dots,w_d) \in \mathbb F_r$ to ensure that the word values are s.a.e., noting that Lemma \ref{lem:counting} easily implies that $d \leq r$ is necessary.
We do not know the precise conditions in case $d=r=2$, but using Lemma \ref{lem:counting}, it follows that simultaneous asymptotic equidistribution can fail for $d=r=3$ even if $w_1,w_2,w_3$ generate a free subgroup of rank $3$. Indeed, consider, $w_1,w_2,w_3 \in \mathbb F_3$, with $\langle w_1,w_2,w_3\rangle = \mathbb F_3 \leq \mathbb F_2 \leq \mathbb F_3.$ In this case the word image is of cardinality at most $(n!)^2$, which (again using Lemma \ref{lem:counting}) is not enough for the image measure $(w_1,w_2,w_3)_*((\mu^{S_n}_{\rm unif})^{\otimes 3})$ to be Kantorovich-Rubinstein close to the equidistribution on $S_n^{3}$. The simplest instance of this phenomenon is the triple $(w_1,w_2,w_3)=(x_1^2,x_2^2,x_1x_2)$, and as far as we can tell, it seems to be the only obstruction.

\begin{conjecture}
Let $d \in \mathbb N$ and $w_1,\dots,w_d \in \mathbb F_r$. If $\langle w_1,\dots,w_d\rangle$ generate a free subgroup of rank $d$ and this group is not contained in a subgroup of smaller rank, then $(w_1,\dots,w_d) \in \mathbb F_r^d$ is simultaneously asymptotically equidistributed. 
\end{conjecture}

In case $r=d$, this conjecture fits well with recent developments in the study of restriction maps in $\ell^2$-cohomology of free groups, see \cite{jaikinrigid}. Indeed, there is a known connection between $\ell^2$-invariants, notions of entropy and the asymptotics of the number of local models that might play a role here, \cite{consh}, however, we do not want to explore this at this stage.

Note that the set of $d$-tuples of words in $\mathbb F_d$, for which the conjecture holds can be identified with a particular subset of ${\rm hom}(\mathbb F_d, \mathbb F_d)$; clearly, this subset contains all automorphisms and is closed under composition. So in order to get started, we provide some examples which are non-surjective. We stick to the case $d=2$, even though the results generalize in an obvious way to the case $d\geq3$.

\begin{proposition} \label{prop:exam}
For $d=2$, the pairs $(x,y^k)$ for $k \neq 0$ and $(x,yxy^{-1})$ are simultaneously asymptotically equidistributed.
\end{proposition}
\begin{proof}
We first treat the pair $(x,y^k)$ with $k\neq 0$.
Let $\mu_n$ be the distribution of $(x,y^k)$ under $(\mu_{S_n}^{\mathrm{unif}})^{\otimes 2}$.
By Lemma \ref{lem:ncycle}, the uniform measure on $S_n$ is at Kantorovich--Rubinstein
distance at most $(\log n +1)/n$ from the uniform measure on the set $\mathcal{C}_n$
of $n$-cycles. Hence, replacing the law of $y$ by the uniform distribution
on $\mathcal{C}_n$ changes $\mu_n$ by at most $O((\log n)/n)$ in
Kantorovich--Rubinstein distance.
Assume first that $(k,n)=1$. Then every $n$-cycle $\sigma$ has a unique $k$-th
root which is again an $n$-cycle. Hence the map
\[
\mathcal{C}_n \to \mathcal{C}_n,\quad \tau \mapsto \tau^k
\]
is a bijection. It follows that $y^k$ is uniformly distributed on $\mathcal{C}_n$. Applying Lemma \ref{lem:ncycle} again shows that this
distribution of $(x,y^k)$ is $o(1)$-close to $(\mu_{S_n}^{\mathrm{unif}})^{\otimes 2}$.

In the other case, if $(k,n) \neq 1$, choose $m \in \{n-1,\dots,n-k\}$ such that $(k,m)=1$. 
Restricting to the set of permutations $\mathcal{C}_{m}$ consisting of an $m$-cycle and $n-m$ fixed points, we reduce to the coprime case. 
On $\mathcal{C}_{m}$, the map $\tau \mapsto \tau^k$ induces a bijection. This proves the claim for $(x,y^k)$ as before.

Now consider $(x,yxy^{-1})$.
Replacing again $x$ by a uniform $n$-cycle up to $o(1)$-error,
we may assume that $x$ is uniformly distributed on $\mathcal{C}_n$.
For fixed $x=\sigma$, the random variable $y\sigma y^{-1}$ is uniformly
distributed on $\mathcal{C}_n$ and independent of $\sigma$.
Hence $(x,yxy^{-1})$ is asymptotically distributed like two independent
uniform $n$-cycles, and the claim follows as above.
\end{proof}

\begin{remark} \label{rem:consequences}
Composing with automorphisms, we obtain that pairs like $(x,[x,y])$ and $(x,yxy)$ are s.a.e., moreover, we also see that if $(x,w(x,y))$ s.a.e., then also $(x,w(x,y)xw(x,y)^{-1}).$ The easiest example which we cannot resolve is $(x,y^2xy^{-1}).$ However, we want to record the following result:
\end{remark}

\begin{theorem}\label{thm:conj-1-implies-sae}
If Conjecture \ref{conj:conj} holds, then the pair $(x,w(x,y))$ is s.a.e.\ for every $w \in \mathbb F_2 \setminus \langle x \rangle$.
\end{theorem}

We will need the following group theoretic lemma.

\begin{lemma}
Let $w \in \mathbb F_2 \setminus \langle x \rangle$. Then there exist $l \ge 0$,
integers $k_1,\dots,k_l \in \mathbb{Z}\setminus\{0\}$, and words
$w_0,\dots,w_l \in \mathbb F_2$ such that
\[
w_i = w_{i-1} x^{k_i} w_{i-1}^{-1} \quad (1 \le i \le l), \qquad w_l = w,
\]
and $w_0$ is not conjugate to a power of $x$.
\end{lemma}

\begin{proof}
Write $w$ in reduced form. If $w$ is not conjugate into $\langle x \rangle$,
we can put $l=0$ and $w_0=w$. Otherwise, $w=v x^k v^{-1}$ for some $v \in \mathbb F_2$ and $k \in \mathbb{Z} \setminus \{0\}$. Moreover, we may assume that $v$ is not ending with $x^{\pm 1}$, otherwise we could shorten the word $v$. Thus the product is already reduced and the length of $v$ is strictly smaller than the length of $w$. Thus, repeating this process finitely many times must terminate and we get a word $w_0$ not conjugate into $\langle x \rangle$ and a chain of words $w_0,w_1,\dots,w_l=w$ with the desired properties.
\end{proof}

\begin{proof}[Proof of Theorem \ref{thm:conj-1-implies-sae}]
Let's prove the claim first in case $w$ is not conjugate to an element of $\langle x \rangle$. We may assume that $x$ takes its value in an $n$-cycle $\sigma \in S_n$. Then, the second variable is distributed as the word values of the word with constants $w(\sigma,y)$. By assumption on $w$, this word is regular, so, assuming Conjecture \ref{conj:conj}, we conclude that the values of $w(\sigma,y)$ are asymptotically equidistributed by Theorem \ref{thm:main}. This proves the claim.

In general, by the preceding lemma, there exists $l\geq 1$ and a chain $w_0,w_1, w_2,\dots, w_l$ of elements of $\mathbb F_2$, numbers $k_1,\dots,k_l \in \mathbb Z \setminus \{0\}$
with $w_i = w_{i-1} x^{k_i} w^{-1}_{i-1}$, for $1 \leq i \leq l$,
$w_l=w$ and such that $w_0$ is not conjugate to a power of $x$ anymore. Thus, the preceding argument applies to $w_0$. The final claim is proven by induction on the index $i$ in the chain above and Remark \ref{rem:consequences}.
\end{proof}

\begin{remark}
In view of the preceding theorem, it would be interesting to know which pairs $(w_1,w_2)$ can be generated by composition of pairs $(x,w(x,y))$ for $w \not \in \langle x \rangle$ and automorphisms of $\mathbb F_2.$ Sean Eberhard pointed out to us in a discussion on MathOverflow\footnote{\url{https://mathoverflow.net/questions/497790/injective-endomorphisms-of-the-free-group-f-2}} that the endomorphism $(x,y)\mapsto ([[x,y],x],[[x,y],y])$ cannot be generated in this way.
\end{remark}

\begin{remark}
We did an extensive computer search for approximate solution to the word equation $\mu = x^2 \sigma x^{-1}$, the first case for which we cannot resolve Conjecture \ref{conj:conj} or even the potentially more basic question whether the word image is $\varepsilon$-dense for $n$ large enough. For random $\mu,\sigma \in S_{50}$, using simulated annealing, we were consistently able to produce consistently permutations $x \in S_{50}$ such $d(\mu,x^2\sigma x^{-1}) \leq 8$. While this is encouraging, the overall evidence supporting Conjecture~\ref{conj:conj} or some extension of Theorem \ref{thm:dense} remains limited.
\end{remark}

\section*{Acknowledgments}

A.T.\ thanks Jan-Christoph Schlage-Puchta for an interesting discussion on the topic. The authors acknowledge funding by the Deutsche Forschungsgemeinschaft (SPP 2026 ``Geometry at infinity'') and thank the colleagues from the SPP for reviewing the article before publication.

\end{document}